\documentclass[11pt]{amsart}

\usepackage{booktabs}                                 
\usepackage[english]{babel}
\usepackage{latexsym}
 \usepackage[utf8]{inputenc}
\usepackage{babel}
\usepackage{fancyhdr}
\usepackage{longtable}
\usepackage{mathrsfs}
\usepackage{geometry}
\usepackage{textcomp}
\usepackage{caption}
\usepackage{lmodern}
\usepackage{mathrsfs}
 \usepackage[T1]{fontenc}
\usepackage[all,cmtip]{xy}
\usepackage{epsfig}
 \usepackage{amsmath,amsfonts,amssymb,amsthm}
\usepackage{graphics}
\usepackage{graphicx}
\usepackage{verbatim}
\usepackage{dsfont}
\usepackage{enumerate}  

\usepackage{pdflscape}
\usepackage{longtable}
\usepackage{booktabs}
\usepackage{colortbl}

\newcommand{\C}{\mathbb{C}}

\newcommand{\Z}{\mathbb{Z}}

\newcommand{\mF}{\mathcal{F}}

\newcommand{\PP}{\mathbb{P}}

\newcommand{\zt}{\widehat{Z}}
\newcommand{\E}{\mathcal{E}}
\newcommand{\mU}{\mathcal{U}}
\newcommand{\mL}{\mathcal{L}}

\newcommand{\hN}{\widehat{N}}

\newcommand{\of}{\mathcal{O}}

\newcommand{\de}{\partial}

\newcommand{\W}{\bigwedge}

\newcommand{\s}{\mathfrak{sl}}

\DeclareMathOperator{\prim}{prim}
\DeclareMathOperator{\Ext}{Ext}

\DeclareMathOperator{\Sym}{Sym}
\DeclareMathOperator{\ddim}{dim}
\DeclareMathOperator{\Gr}{Gr}

\DeclareMathOperator{\HP}{HP}

\DeclareMathOperator{\van}{van}

\newtheorem{thm}{Theorem}[section]
\newtheorem{corollary}[thm]{Corollary}

\newtheorem{lemma}[thm]{Lemma}
\newtheorem{proposition}[thm]{Proposition}
\newtheorem{definition}[thm]{Definition}

\newtheorem{conj}[thm]{Conjecture}

\newenvironment{sistema}
{\left\lbrace\begin{array}{@{}l@{}}}
{\end{array}\right.}

\makeatletter    
\def\l@subsection{\@tocline{1}{0,2pt}{2pc}{8mm}{\ \ }} 
\makeatletter    
\def\l@section{\@tocline{1}{0,2pt}{2pc}{8mm}{\ \ }} 

\author{Enrico Fatighenti }
\address{Institut de math\'ematiques de Toulouse \\
   Universit\'e Paul Sabatier\\
118 route de Narbonne, 31062 Toulouse, France}
\email[E.~Fatighenti]{efatighe@math.univ-toulouse.fr}

\author{Giovanni Mongardi}
\address{Dipartimento di Matematica \\
Alma Mater Studiorum - Università di Bologna\\
Piazza di Porta San Donato 5, 40126 Bologna, Italia}
\email[G.~Mongardi]{giovanni.mongardi2@unibo.it}
 
\title[Griffiths ring for C.I. in Grassmannians]{A note on a Griffiths-type ring for complete intersections in Grassmannians}

\begin{document}
\maketitle
\begin{abstract}
We calculate a Griffiths-type ring for smooth complete intersections in Grassmannians. This is the analogue of the classical Jacobian ring for complete intersections in projective space and allows us to explicitly compute their Hodge groups.\end{abstract}

\section{Introduction}
Griffiths' theory of residues is a powerful tool in algebraic geometry. It identifies the Hodge groups of a smooth projective hypersurface $X$ with some special homogeneous slices of a graded ring, the \emph{Jacobian ring} associated with the defining equation of $X$ (see \cite{griffiths} for the original result). Its very explicit nature has led to proofs of several well-known  theorems, for example, the Torelli theorem or the Noether-Lefschetz theorem in some special cases, including e.g. threefolds.

This result has been generalised to the case of complete intersection in toric varieties, thanks to the work of Batyrev and Cox, Dimca, Konno, Mavlyutov, and many others, (\cite{cox}, \cite{dimca1994residues}, \cite{konno}, \cite{mav}). In this case, the generalised Jacobian ring is as explicit as in the hypersurface case (given in terms of generators and relations). Another generalisation was subsequently given by Green in \cite{green}, who investigated the case of hypersurfaces of sufficiently high degrees in an arbitrary variety. However, the latter was less explicit than the former case. In a very recent work which was developed in parallel with ours, Huang--Lian--Yau--Yu (\cite{yau}) generalised Green's description to zero loci of homogeneous vector bundles. Their approach and ours share many techniques, which were also used in the first author's PhD thesis (\cite{thesis}). However, they differ both in the scope of the results, which hold in greater generality in \cite{yau}, and in the explicit computability of Hodge structures, which is explained in greater detail in the present paper.
 In a preliminary chapter, we go through a brief summary of some of the literature in the topic.  
The purpose of this paper is to construct explicitly a Jacobian-type ring for complete intersections in Grassmannians. 
In particular, in what we consider to be the core result, we give a presentation (in terms of generators and relations) of the \emph{Griffiths ring} that plays the role of the classical Jacobian ring in this context. In particular, we produce a simple recipe to explicitly write down the Jacobian-type ideal, echoing  the original spirit of the work of Griffiths. We present as well several meaningful examples in Sections 3.1, 4.1 and 4.2 to show in detail how the computations can be done. These computations can be easily replicated for any other example, either by hand or by using computer--algebra software such as Macaulay2 \cite{Mac2}.

 When $X_d \subset \Gr(k,n)$ is a smooth hypersurface we define the \emph{Griffiths ring} $R^G_f$ as in Definition \ref{jacobian} and $I_{p-1,p}$ as in Definition \ref{contr}. Our first result is the following.

\begin{thm}\label{residueGrass} Let $X_d$ be a smooth hypersurface in the Grassmannian $G=\Gr(k,n)$. Set $N:=\ddim( G)=k(n-k)$, and define $R^G_f$ the Griffiths ring for $X$ as defined in Definition \ref{jacobian}. Assume that $d \geq n-1$. If $\ddim(X) =N-1 \equiv 0 \ (2)$, then $$ [R_f^G]_{(p+1)d-n} \cong H^p_{\van}(X, \Omega^{N-1-p}).$$
If $\ddim(X) =N-1 \equiv 1 \ (2)$, then $$ [R_f^G]_{(p+1)d-n} \cong H^p_{\van}(X, \Omega^{N-1-p}) \oplus \delta_{p, \frac{N}{2}}I_{p-1,p},$$ where $\delta_{p, \frac{N}{2}}$ is the Kronecker delta symbol.
\end{thm}

We discuss the case $d \leq n-2$ and give some explicit formulae for the Grassmannian of lines as well. \\
In the case of $Z$ a complete intersection $Z=Z_{d_1,\ldots, d_c} \subset \Gr(k,n)$, we use a Cayley trick to produce a specific version of the Griffiths ring $\mU$ as in \eqref{griffithsring1}.  Set $m:= \sum d_i -n$ , so that $\omega_Z \cong \of_Z(m)$. Our result is then the following.

\begin{thm} \label{grci} Let $Z$ be a smooth complete intersection in a Grassmannian Gr(k,n), and let $\mU$ be the Griffiths ring attached to $Z$. Suppose $m \geq -1$. Then if $ \ddim(Z)=N-c$ is even $$ \mU_{p, m} \cong H^{N-c-p,p}_{\van}(Z).$$ If  $ \ddim(Z)=N-c$ is odd $$ \mU_{p, m} \cong H^{N-c-p,p}_{\van}(Z) \oplus \delta_{p, \frac{N-c}{2}}I_{p,p-1}(G).$$
\end{thm}
We give several significative examples, such as Fano 5-folds and 4-folds of genus 6 and degree 10, and a Calabi-Yau section of the Grassmannian $\Gr(2,7)$. We conclude with an appendix on Fano varieties of K3 type: we intend this as the beginning of a classification project that we plan to develop in a series of future works.
\subsubsection*{Notation}
If $V_n$ is a $\C$-vector space of dimension $n$, we denote by $\Gr(k,V_n)=\Gr(k,n)$ the Grassmannian of k-planes in $V_n$ (sometimes in longer formulae -- and whenever there is no risk of confusion -- we will denote it by $G$). We will denote by $N:=k(n-k)$ the dimension of $\Gr(k,n)$. Denote by $\of_G(1)$ the ample generator of $\mathrm{Pic}(\Gr(k,n)) \cong \Z$: we have then $\omega_G \cong \of_G(-n)$.  We will always assume that $k \neq 1, n-1$: this is motivated by substantial differences in the theory between projective space and Grassmannian varieties.\\
Denote by $$S= \bigoplus_{a \geq 0} S_a, \ \ S_a=H^0(\Gr(k,n), \of_G(a))$$ the homogeneous coordinate ring of the Grassmannian; in particular, a hypersurface $X$ of degree $d$ will be given by the vanishing of a $f \in S_d$.
Throughout the whole paper, $X_{d_1, \ldots, d_c} \subset \Gr(k,n)$ will denote a complete intersection of multidegree $(d_1, \ldots, d_c)$ in the Grassmannian $\Gr(k,n)$. All varieties are assumed to be smooth and projective. We work over $\C$.


\subsubsection*{Acknowledgments}
The authors wish to thank Enrico Arbarello for useful comments and suggestions, Camilla Felisetti and Luca Migliorini for their support.
Many of these results appeared as well in the PhD thesis of the first author (\cite{thesis}). We would also like to thank Miles Reid, Christian B\"ohning and Alessio Corti for discussions, insights and suggestions. EF was supported by MIUR-project FIRB 2012 ``Moduli spaces and their applications.'' GM was supported by ``Progetto di ricerca INdAM per giovani ricercatori: Pursuit of IHS.'' Both authors are member of the INDAM-GNSAGA.\\ 
While completing this paper, we learned that An Huang, Bong Lian, Shing-Tung Yau, and Chenglong Yu obtained similar results in an independent way; see \cite{yau} and the subsequent \cite{yau2}. We believe that the two papers together complete each other and we invite the interested reader to check both of them. We thank in particular Shing-Tung Yau for his nice comments on our work.
\section{Preliminaries}
\subsection{Original Griffiths theory and link with deformations of affine cones}
If $X$ is a smooth projective hypersurface, Griffiths' theory of residues explicitly determines the Hodge structures of $X$ in terms of the coordinate ring of $X$. The result is classic. For a general overview, we refer to \cite[6.10]{voisin2}.
\begin{thm} \label{res} The $n$-dimensional vanishing
    Hodge structure of a degree $d$ smooth projective hypersurface $X \subset
    \PP^{n+1}$ is given by the isomorphism 
    \[
    H^{n-p+1, p-1}_{\van}(X) \cong (\C[x_0, \ldots, x_{n+1}]/ J_f)_{pd-n-2}
    \]
    where $J_f$ is the ideal spanned by all the partial derivatives
    $(f_0, \ldots, f_{n+1})$  of $f$.
  \end{thm}
  
 The ideal $J$ is often called in the literature the \emph{Jacobian ideal}. The notation $H^n_{\van}(X)$ (and similarly for the $(p,q)$ part) will denote the vanishing subspace of the cohomology group, see \cite[2.3.3]{voisin2} for a definition. In the literature, the result is often phrased in terms of \emph{primitive cohomology}. This is because for a smooth projective hypersurface primitive and vanishing cohomology agrees. The \emph{Jacobian ring} $\C[x_0, \ldots, x_{n+1}]/ J_f$ of a smooth projective hypersurface coincides moreover with  $T^1_{A_X}$, the infinitesimal first-order deformation module of the affine cone $A_X$ over $X$. The latter is a classical object in algebraic geometry and deformation theory. We refer to \cite{io} for an overview of its properties that are needed in the present paper. When $X$ is an arbitrary smooth projective variety of codimension $c$ there is a link between the deformations of $A_X$ and the Hodge theory of $X$. An assumption needed for the result is the subcanonicality property for $X$, that is we can write $\omega_X$ as $\of_X(c)$, some $c \in \Z$.
\begin{thm}[Theorem 1.1 in \cite{io}] Let $X$ be a smooth projectively normal variety of dimension $n>1$, and let $m\in \mathbb{Z}$ be the integer such that $\omega_X \cong \of_X(m)$. If $H^1(X, \of_X(k))=0$ for every $k\in \mathbb{Z}$, then we have \[ (T^1_{A_X})_{m} \cong H^{n-1,1}_{\prim}(X). \]
\end{thm}
The relation between $T^1$ and Hodge theory will be crucial for the rest of this paper.

\subsection{Cohomology of projective bundles and of complete intersections in $\PP^N$}
In order to extend Griffiths' original result from hypersurfaces to complete intersections in projective spaces, one of the main tools is the Cayley trick approach of Dimca, Konno, Terasoma et al. Starting from a complete intersection $Z \subset P$ one can construct a hypersurface $\zt$ in a projective bundle $\PP(\E)$ over $P$. The middle vanishing cohomology of $Z$ and  of $\zt$ coincide up to a shift. Therefore one can apply the construction for a hypersurface to $\zt \subset \PP(\E)$, with suitable modifications. In \cite{konno} a generalised version of a Griffiths ring for a variety defined by the zero set of a generic section of $\E$ is defined. However, the result was made explicit only for complete intersections in a projective space. We give here an explicit version of Griffiths' residue theorem for complete intersections in Grassmannians as well.  We will now go through a recap of some preliminary concepts on projective bundles that we will need later in the paper.

Let $\E$ be a vector bundle of rank $c$ on an $n$ dimensional smooth and proper variety $X$. Denote by $\E_x$ the fiber over $x \in X$. Consider the projective vector bundle $$ \pi: Y = \PP(\E) \to X,$$ whose fiber is the space $\PP(\E_x)$. To avoid any confusion we consider $\PP(\E)$ as the space of rank one quotients. Useful sequences to understand the geometry of $Y$ in terms of $X$ are the \emph{relative tangent sequence}
\begin{equation} \label{reltang}
 0 \to T_{Y/X} \to T_Y \to \pi^* T_X \to 0,
\end{equation} and the \emph{relative Euler sequence}
\begin{equation}
0 \to \of_Y \to \pi^* \E^* \otimes \mL \to T_{Y/X} \to 0,
\end{equation}  where $\mL = \of_Y(1)$ denotes the tautological quotient line bundle on the projective bundle $Y$, which is ample if and only if $\E$ is. The following lemma is useful in this context.
\begin{lemma}[Lemma 1.2, \cite{konno}] \label{pistar}Let $\mF$ be a vector bundle on $X$. Then $$H^q(Y, \pi^*\mF \otimes \mL ^h) \cong \begin{sistema}   H^q(X, \mF \otimes \Sym^h \E) \ \ \textrm{if } h\geq 0 \\ H^{q-c+1}(X, \mF \otimes det \ \E^* \otimes \Sym^{-h-c}\E^*) \ \ \textrm{if }h \leq -c  \\ 0 \ \ \textrm{              otherwise}\end{sistema}$$
\end{lemma}
From the above Lemma it follows that $H^0(X, \E) \cong H^0(Y, \mL)$. To $\sigma \in H^0(X, \E)$ we associate the corresponding section  $\hat{\sigma}$ of $\mL$ on $Y$.
Let $Z$ and $\zt$ be the zero loci of $\sigma$ and $\hat{\sigma}$. 
 The Hodge theory of $Z$ and $\zt$ are strongly related: namely, we have the following result.
\begin{proposition} [Proposition 4.3, \cite{konno}]\label{konno} There exists a canonical isomorphism of Hodge structures $$H^q_{\van}(Z, \C) (1-c) \cong H^{q+2c-2}_{\van}(\zt, \C).$$
\end{proposition}

As an example, take $X \cong \PP^N$. Assume that $\E$ splits into a direct sum of line bundles, i.e. $\E \cong \of_{\PP^n}(d_i)$. Since we have $H^0(X, \E) \cong H^0(Y, \mL)$ we can consider the total coordinate ring of $Y$ $$ S= \C[x_0, \ldots, x_N, y_0, \ldots, y_c].$$
The Picard group of $\zt$ has rank two: therefore the ring above comes with a suitable bi-grading. We set deg$(x_i)=(0,1)$ and deg$(y_i)=(1,-d_i)$. This choice of bi-grading is inspired by the above isomorphism with the global sections of the normal bundle $\E|_Z$.
We have the following result.
\begin{thm}[Theorem 7 in \cite{dimca1994residues}] \label{dimcares}Let $Z=V(f_1, \ldots, f_c)$ be a smooth complete intersection of dimension $n$ in $\PP^N$ with normal bundle $ \bigoplus^c \of(d_i)$ and $\omega_Z \cong \of_Z(m)$. Set $F= \sum f_i y_i$. Denote by $$\mU_{a,b}:=(S/J)_{a,b},$$ where $J$ is the ideal generated by $(\frac{\de F}{\de x_0}, \ldots, \frac{\de F}{\de y_c})$. Then $$ \mU_{p,m} \cong H^{n-p,p}_{\van}(X).$$
\end{thm}

\section{Hypersurfaces in Grassmannians}
 The first step in our analysis consists of formulating the following definition. In the definition below we use the fact that $ \mathfrak{sl}_n$ acts on $S$. We will properly define this action in the lines that follow.
\begin{definition}[cf. \cite{green}, \cite{saito}] \label{jacobianringdef}The \emph{generalised Jacobian ideal} or \emph{Griffiths ideal} $J_f$ of a smooth hypersurface $X=V(f)$ in Gr(k,n) is the homogeneous ideal of $S$ generated by $f \in S_d$ and  $$\lbrace v \cdot f \ | \ v \in \mathfrak{sl}_n \cong H^0(G, T_G) \rbrace .$$
We denote by $R^G_f=S/J_f$ the corresponding \emph{Griffiths ring}.
\end{definition}
We want now to introduce an equivalent definition of $R^G_f$ in the most possible explicit way, that is in terms of generators and relations. Let us fix a basis $v_1, \ldots, v_n$ for $V_n$ and a dual basis $x_1, \ldots, x_n$ for $V_n^{\vee} \cong \C[x_1, \ldots, x_n]_1$.  It is well known that $$H^0(G, \of_G(1)) \cong \W^k V_n^{\vee} \cong \langle \cdots ,x_I, \cdots \rangle,$$ where $I$ denotes a multi-index $\lbrace i_1,\cdots, i_k\rbrace$ of $\lbrace 1, \cdots, n \rbrace $ of length $k$, with $i_1 < \cdots <i_k$ and $x_I:= x_{i_1} \wedge \ldots \wedge x_{i_k}.$ In particular, $S$ is isomorphic to the Pl\"ucker algebra $S\cong \C[x_I]/P,$ with $x_I$ is as above, with $P$ denoting the ideal generated by the equations of the Pl\"ucker embedding. These can be computed quite easily in a recursive way, for example using Macaulay2.\\
To have a complete understanding of $R^G_f$ we only have to make the $\mathfrak{sl}_n$ action explicit. There is a canonical action of $\mathfrak{sl}_n$ on the dual of its tautological module $(V_n)^{\vee}$ (cf. \cite{popov}). Recall that $\mathfrak{sl}_n$ is generated by $$\lbrace  E_{i,j}, \ E_{i,i}- E_{j,j} \ | \ i,j=1, \ldots, n, i \neq j \rbrace, $$
where $E_{i,j}$ denotes the matrix with one in the $(i,j)$-place and zeroes elsewhere. $E_{i,j}$ acts on $(V_n)^{\vee}$ as a differential operator: more precisely to $E_{i,j}$ corresponds the derivations $D^i_j$ defined by $$D^i_j= x_i \frac{\de}{\de x_j}.$$ The action of $D^i_j$ induces a natural action on $\W^k V^{\vee}$ and on 

$ \Sym^r \W^k V^{\vee}$ simply by Leibniz's rule.\\ Therefore, if $X \subset \Gr(k,n)$ is given by the vanishing of a polynomial $f \in S_d$, $J$ will be generated by $f$ itself and by the $n^2-1$ degree $d$ polynomials given by  \begin{equation}\label{genjacobi}\lbrace D^i_j(f), \ D^i_i(f)- D^j_j(f) \ | \ i,j=1, \ldots, n, i \neq j \rbrace. 
\end{equation}
 We can then rephrase the definition of the Griffiths ring as follows. 
\begin{definition}\label{jacobian}Let $X=V(f)$ be a smooth hypersurface in the Grassmannian Gr(k,n). Let $S$ be the coordinate ring of the (affine cone over the) Grassmannian, and let $J$ be the ideal of $S$ generated by $f$ and the equations in \eqref{genjacobi}. We define the Griffiths ring of $X$ as
$$R^G_f:=S/J.$$
\end{definition}
The above definition is quite similar to the one given in Theorem \ref{res} in the projective case. The main difference is the $\mathfrak{sl}_n$-action, which is different from the usual one for $\PP^n$ that sends $ f \mapsto \frac{\de f}{\de x_i}$. Our assumption $k\neq (1, n-1)$ in $\Gr(k,n)$ is therefore relevant.

Generalising Griffiths' calculus, when appropriate vanishings are provided, the Hodge groups $H^p_{\van}(\Omega^{n-p})$ are contained in (some specific homogeneous component of) $S$. In particular, there is a surjective map of graded rings $$\bigoplus S_a \longrightarrow \bigoplus H^{p,n-p}_{\van}(X).$$
Our purpose is to identify the kernel of this surjective map with the above-defined Jacobian ideal $J_f$. Moreover, in what we consider being the core result of this section, we show how to give an explicit presentation of the Jacobian ring (and its graded components) in terms of generators and relations. This in turn allows us to recover explicit (polynomial) basis for the Hodge groups $H^{p,q}_{\van}(X)$, in a generalisation of Griffiths' theorem on $\PP^n$.\\
We point out that the required vanishings for Griffiths' theorem to hold do not always work in the Grassmannian case. Nevertheless, we give a generalised version of Griffiths' strategy, showing how to effectively use our result in a few distinguished examples.\\
The first step consists in linking the generalised Jacobian ring to the $T^1_{A_X}$ of the affine cone over $X$. Recall from \cite{schlessinger} that $T^1_{A_X}$ can be defined as $\Ext^1(\Omega^1_{A_X}, \of_{A_X})$ under the assumption of projective normality of $X$. We refer to \cite{io} for a collection of properties relevant in this context. In particular, recall that for a smooth projective hypersurface the module $T^1_{A_X}$ has a ring structure, and it is isomorphic, up to a degree shift, to the classical Jacobian ideal of $X$. We want to show that the same happens for hypersurfaces in Grassmannian, with the appropriate definition of the Griffiths ring given above. In what follows, recall that $N:=k(n-k)$ denotes the dimension of $\Gr(k,n)$ and that $\omega_G \cong \of_G(-n)$.

\begin{lemma}\label{t1jacobi} Let $X$ be a smooth hypersurface of degree $d$ in the Grassmannian $G=\Gr(k,n)$ defined by the vanishing of $f \in H^0(G, \of_{G}(d))$. We have an isomorphism $$T^1_{A_{X_d}}(-d) \cong R^G_f.$$ 
\end{lemma}
\begin{proof}
Consider the short exact sequence $$ 0 \to T_{X} \to T_G|_{X} \to \of_X(d) \to 0.$$
For any twist with $\of_X(h)$ we consider the associated long exact sequence in cohomology on $X$
\begin{equation} \label{h1tg} 
H^0(TG|_X(h)) \stackrel{\beta}{\to} H^0( \of_X(d+h)) \stackrel{\alpha}{\to} H^1(T_X(h)) \to H^1(TG|_X(h)).
\end{equation}
The first thing to show is the vanishing of the last term in the sequence above. One uses the two standard exact sequences (for any $k,t$)
\begin{equation}\label{seq1h}
0 \to \Omega^k_G(t) \to \Omega^k_G (t+d) \to \Omega^k_G|_X(t+d) \to 0,
\end{equation} 
\begin{equation}\label{seq2h}
 0 \to \Omega^{k-1}_X(t) \to \Omega^k_G|_X(t+d) \to \Omega^k_X(t+d) \to 0,
 \end{equation}
and the fact that by Serre duality $H^1(X, T_G(h)|_X) \cong (H^{N-2}(\Omega^1_G|_X (-n+d-h))^{\vee}$. Indeed the latter is zero after expanding in cohomology the first sequence since by Borel--Bott--Weil theorem we have the vanishing of $H^q (\Omega^1_G(t))$ for any $(q,t) \neq (1,0)$, $q>1$.
Therefore in \eqref{h1tg} by properties of exact sequences one has $$H^1(X, T_X(h)) \cong H^0(X, \of_X(h+d))/ \textrm{Im}(\beta).$$
On the other hand, the action of $H^0(TG|_X) \cong \mathfrak{sl}_n$ is given as the derivation  action of $\mathfrak{sl}_n$ on the space of homogeneous polyonomial of degree $h$ in the coordinate ring. For every $h$ therefore $\beta$ coincides with the action defined in \eqref{genjacobi}. The right hand side of the isomorphism above coincides with the given definition of the Jacobian ring $R^G_f$. For dimension reasons $H^2(\of_X(h))=0$.
 This implies that $$T^1_{A_{X}}(-d+h) \cong H^1(T_X(h)).$$
\end{proof}
The above lemma gives us almost everything we need. In fact, thanks to this result we can work directly on $T^1_{A_X}$, whose graded components are identified with the cohomology groups of twists of the tangent bundle $T_X$. We can therefore apply all original Griffiths' machinery, proving at the same time the results for the Griffiths ring $R^G_f$. Many of the proofs use standard diagram-chasing techniques, and therefore we will just sketch them.\\
Using \cite[Theorem 1.1]{io}, Lemma \ref{t1jacobi} implies  $$(R^G_f)_0 \cong H^1(T_X), \ \ (R^G_f)_{m} \cong H^{n-1,1}(X),$$
where $m$ is the integer such that $\omega_X \cong \of_X(m)$. These cohomology spaces both coincide with their primitive part, since $H^2(X, \of_X(k))=0$ for any $k$. In the case of projective hypersurfaces Griffiths' theory implies $$H_{\van}^{n-p,p}(X) \cong (T^1_{A_X})_{(p-1)d-m} \cong H^1(T_X((p-1)d-m),$$
see \cite[Corollary 3.13]{io}. These spaces can be shown to be isomorphic a priori, without deducing it from the previous theorem. This is implied by the vanishings of $H^q(\Omega_{\PP}^p(k))$ for $p\geq 0,q,k>0$ by Bott's theorem (and Hard Lefschetz theorem). On the Grassmannian Gr(k,n) the vanishing of the cohomology group of twisted differentials is a more subtle question. Borel-Bott-Weil theorem is the main source to address the computations of these cohomology groups. A classical survey can be found  for example in Snow's paper  \cite{snow}. The following lemma provides the vanishings required in the Grassmannian case.
\begin{lemma}\label{lemmavanishing}Let $X \subset \Gr(k,n)$ be a smooth hypersurface of degree $d$ and take $p \in \lbrace 1, \ldots, N-2 \rbrace$. Suppose that the following vanishings hold:
\begin{enumerate}[(I)]
\item $H^{p-1}(\Omega_{G}^{N-p}(d))=0$;
\item $ H^p(\Omega_{G}^{N-p}(d))=0$;
\item $H^p(\Omega_G^{N-p})=0$;
\item $H^{p+1}(\Omega_G^{N-p})=0$.
\end{enumerate}
Then the following isomorphism holds
$$ H^{p-1} (\W^{p-1} T_X(2d-n)) \cong H^p(\W^p T_X(d-n)).$$
\end{lemma}

\begin{proof}
Consider the tangent-normal sequence raised to the $p$-th power
$$0 \to \W^p T_X(d-n) \to \W^pT_G|_X(d-n) \to \W^{p-1}T_X(2d-n) \to 0.$$
The long associated sequence in cohomology is \begin{small}
$$ \cdots \to H^{p-1}( \W^p TG|_X(d-n)) \to H^{p-1}(\W^{p-1}T_X(2d-n)) \to  H^{p}(\W^{p}T_X(d-n)) \to H^{p}( \W^p TG|_X(d-n)) \to \cdots $$
\end{small}
By the standard tangent pairing $$H^{p-1}(\W^p T_G|_X(d-n)) \cong H^{p-1}(\Omega^{N-p}_G|_X (d));$$
$$H^{p}(\W^p T_G|_X(d-n)) \cong H^{p}(\Omega^{N-p}_G|_X (d)).$$
Using the Koszul complex one has that the vanishing conditions (I,III) imply the vanishing of $ H^{p-1}(\Omega^{N-p}_G|_X (d))$, and the same with $H^{p}(\Omega^{N-p}_G|_X (d))$ and conditions (II, IV).
\end{proof}

The above Lemma gives us only one step of the iterated multiplication map. We remark that for every step the latter is the connecting homomorphism in cohomology of the $p$-th wedge power of the normal sequence. However, one can replicate the same technique and get even more conditions. The proof is rather technical and we will omit it, since it follows the same lines of Lemma \ref{lemmavanishing}.
\begin{lemma} \label{multivan}
Let $X \subset \Gr(k,n)$ be a smooth hypersurface of degree $d$ and consider $p \in \lbrace 1, \ldots, N-2 \rbrace$. Suppose that the following vanishings hold
$$H^j(\Omega^{q}_G(ld))=0, \  \ q=N-p, \ldots, N, j= p+1, \ldots, 1, \ l=0, \ldots, p.$$
Then the following isomorphism holds
$$ H^1(T_X(pd-n)) \cong H^p(X, \Omega^{N-1-p}_X).$$
\end{lemma}
We point out that this set of vanishings is slightly stronger than the one we need. As an example, see Lemma (\ref{lemmavanishing}), where the vanishing of $H^{p+1}(\Omega_G^{N-p}(d))$ is not needed.
What we have to understand now is for which $X_d\subset \Gr(k,n)$ the vanishing conditions of Lemma \ref{multivan} are automatically satisfied. Borel--Bott--Weil theorem transforms the vanishing question into a combinatorial one. We quote the following result by Snow (\cite{snow}).
\begin{thm}[Thm. 3.2, 3.4, 3.5 in \cite{snow}] \label{snow}Denote by $G$ the Grassmannian $\Gr(k,n)$. Then $H^p(G, \Omega^q(t))=0$, for $t\geq 1$, if any of the following conditions are satisfied:
\begin{enumerate}[(I)]
\item $t\geq n$;
\item $kp \geq (k-1)q >0;$
\item $p >N-q$;
\item $q> N-k$, $(k,n) \neq (2,4)$;
\item $q \leq t$, with $p>0$;
\item $t \geq n-k$ and $p >\frac{(n-1-t)(n-t)}{2}$.
\end{enumerate}
\end{thm} 
We are now in position to prove the main result of this section. We recall first the description of the Hodge groups of the Grassmannian. Since the Grassmannian is a homogeneous variety, $h^{i,j}(G)=0$ for $i \neq j$. On the other hand, when $i=j$ the dimension of these spaces are
$$h^{j,j}(G)= \#\lbrace (a_1,\ldots, a_k) | n-k \geq a_1 \geq \ldots \geq a_k \geq 0, \ \sum a_i=j \rbrace. $$ We will need the following definition. 
\begin{definition} \label{contr} For $j \leq \frac{N}{2}$ define $ I_{j-1,j}$ as the cokernel of the injective map given by the multiplication by a hyperplane class $$0 \to H^{j-1,j-1}(G) \to  H^{j,j}(G).$$
\end{definition}
\begin{thm}\label{residueGrass} Let $X_d$ be a smooth hypersurface in the Grassmannian $G=\Gr(k,n)$, and $R^G_f$ the Jacobian ring for $X$ defined in Definition \ref{jacobian}. Assume that $d \geq n-1$. If $\ddim(X) =N-1 \equiv 0 \ (2)$, then $$ [R_f^G]_{(p+1)d-n} \cong H^p_{\van}(X, \Omega^{N-1-p}).$$
If $\ddim(X) =N-1 \equiv 1 \ (2)$ then $$ [R_f^G]_{(p+1)d-n} \cong H^p_{\van}(X, \Omega^{N-1-p}) \oplus \delta_{p, \frac{N}{2}}I_{p-1,p},$$ where $\delta_{p, \frac{N}{2}}$ is the Kronecker delta symbol.
\end{thm}
\begin{proof} By Lemma \ref{t1jacobi} one has $$(R^G_f)_{k+d} \cong (T^1_{A_X})_k \cong H^1(X, T_X(k)).$$
In particular, thanks to Lemma \ref{multivan} we will have $$(R^G_f)_{pd-n+d} \cong  H^1(X, T_X(pd-n)) \cong H^p(X, \Omega^{N-1-p}),$$ provided that the vanishing conditions in the hypotheses hold. By Theorem \ref{snow} all these vanishings are automatically satisfied if $d \geq n$ (part I) and if $d=n-1$ (part VI) except possibly $H^{i+1}(\Omega_G^{N-1-p})=H^{i+1}(\Omega_G^{N-1-p})=0$.\\
Thanks to the given description of the cohomology ring of the Grassmannian, we know that the above groups vanish for almost all values of $i$. In particular, from \eqref{seq1h} and \eqref{seq2h} one gets the sequence in cohomology $$0 \to H^{N-1-p,p-1}(G) \to H^{N-p,p}(G) \to H^{p-1}(\Omega^{N-p}_X(d)) \to H^{N-1-p,p}(X) \to H^{N-p,p+1}(G) \to 0, $$ where we have already taken into account all the other vanishings of Lemma \ref{multivan}. The Hodge groups in the Grassmannian will vanish unless $p+1=N-p$ or $p=N-p$. In the first case $\ddim(X)= N-1=2p$ is even, and we have $$0 \to H^{p-1}(\Omega^{N-p}_X(d)) \to H^{N-1-p,p}(X) \to H^{N-p,p+1}(G) \to 0, $$ that is $$H^1(T_X((\frac{N-1}{2}d-n)) \cong H^{\frac{N-3}{2}}(\Omega^{\frac{N+1}{2}}_X(d))  \cong H^{\frac{N-1}{2}, \frac{N-1}{2}}_{\van}(X).$$ 
When the dimension of $X$ is odd, we first remark that $H^{N-1}(X) = H^{N-1}_{\van}(X)$.
We have $N=2p$ and  $$0 \to H^{N-1-p,p-1}(G) \to H^{N-p,p}(G) \to H^{p-1}(\Omega^{N-p}_X(d)) \to H^{N-1-p,p}(X) \to 0, $$ that is 
$$H^1(T_X((\frac{N}{2}d-n)) \cong H^{\frac{N}{2}-1}(\Omega^{\frac{N}{2}}_X(d))  \cong H^{\frac{N}{2}-1, \frac{N}{2}}(X) \oplus I_{\frac{N-2}{2},\frac{N}{2}}.$$
\end{proof} 
We notice how the statements of the above theorem directly generalise Theorem \ref{res}. In particular, the degree of the homogeneous slices in the Griffiths ring representing the Hodge groups are determined by the degree of the hypersurface and the canonical class of the ambient space.

The above theorem guarantees an extension of the Griffiths' residue calculus to all but a finite number of cases for any Grassmannian (namely, in the Fano case of index $>1$). Of course Borel-Bott-Weil theorem can be effectively used to get either more vanishings or to easily compute the exceptions to the above result in the Fano case. As we have seen, in general for a Grassmannian $\Gr(k,n)$ the difference between $(R_f)_{(p+1)d-n}^G$ and $H^{p,n-p}_{\van}(X)$ can be computed in terms of $H^p(\Omega_G^q(k))$. There exist ad-hoc formulae for these groups, but a general statement is complicated to find. The situation is slightly better for the Grassmannian of lines $\Gr(2,n)$. 

\begin{corollary}Let $X$ be a smooth hypersurface of degree $d$ in the Grassmannian Gr(k,n) defined by $f \in H^0(\of_G(d))$. Then
$$ \bigoplus_{p=1}^{N-1} (R_f)_{(p+1)d-n} \oplus B_{N-1-p,p} \cong \oplus (H^{N-1-p,p}_{\van}(X) \oplus A_{N-1-p,p})$$ with the possible residual contributions $A_{p,N_p-1}, B_{p,N_p-1}$ determined by the non-vanishing of the groups in Lemma \ref{multivan} and therefore depending only by some residual cohomologies of $H^j(\Omega^q_G(k))$.
\end{corollary}

\begin{corollary}\label{gr2n}Let $X$ be a smooth hypersurface of degree $d$ in the Grassmannian Gr(2,n) defined by $f \in H^0(\of_G(d))$. Then
$$ \bigoplus_{p=1}^{n-1} (R_f)_{(p+1)d-n} \cong \oplus H^{N-1-p,p}_{\van}(X) \oplus \delta_{p, \frac{N}{2}}I_{p-1,p}$$ with the possible exceptions of
$$\ p=\frac{2n-1-d}{3} \textrm{ and } \ p=\frac{4n-9-d}{3}.$$
\end{corollary}
\begin{proof}
The case $d\geq n-1$ is already addressed by Theorem \ref{residueGrass}. Therefore we just need to check the case $d \leq n-2$. In particular, we need to check the vanishing of the groups in Lemma \ref{multivan}. For the Grassmannian of lines however, these are listed in \cite[Lemma 0.1]{peternell}.  
\end{proof}
We want now to provide an example to show to the reader how our method can be effectively used in computations.

\subsection{A worked example: Fano fivefold of degree 10}

Our first example is a smooth quadric fivefold hypersurface in the Grassmannian Gr(2,5).  The Grassmannian $\Gr(2,5)$ has dimension six, and it is embedded under the Pl\"ucker embedding in $\PP^9= \PP(\W^2 V_5)$. Its description is well-known, but we will briefly recall it for the convenience of the reader. Its homogeneous ideal of relations is given by the submaximal Pfaffians of a generic skew 5 by 5 matrix, and we can write the five equations as $$I_G=(x_{3,4}x_{2,5}-x_{2,4}x_{3,5}+x_{2,3}x_{4,5},x_{3,4}x_{1,5}-x_{1,4}x_{3,5}+x_{1,3}x_{4,5},x_{2,4}x_{1,5}-x_{1,4}x_{2,5}+x_{1,2}x_{4,5},$$$$
     x_{2,3}x_{1,5}-x_{1,3}x_{2,5}+x_{1,2}x_{3,5},x_{2,3}x_{1,4}-x_{1,3}x_{
     2,4}+x_{1,2}x_{3,4}).$$

 As before, we think of $X$ as defined by the vanishing of an (appropriate) single polynomial $f$ in $H^0(\of_G(2))$. The hypersurface $X$ is an example of a Gushel-Mukai variety in the sense of \cite{ilievmanivel}. We point out that the novelty of this computation is not the determination of the Hodge numbers (that were computed before for example in \cite{nagel}), but the fact that we can now determine explicitly generators and relations for these groups.
 
  By adjunction formula one has that $\omega_X \cong \of_X(-3)$. $X$ is a Fano fivefold of degree 10 and genus 6. In particular, we know straight away that $$H^{0,5}(X) \cong H^{5,0}(X) \cong H^0(K_X)=0.$$  
\begin{lemma}\label{gmhodge} Let $X$ be as above. The following isomorphisms hold \begin{itemize} \item $   (R^G_f)_{-1} \cong H^1(T_X(-3)) \cong H^{4,1}(X)$; \item  $(R^G_f)_{1}\cong H^1(T_X(-1)) \cong H^{3,2}(X)$;\item$ (R^G_f)_{3}\cong H^1(T_X(1)) \cong H^{2,3}(X);$\item $(R^G_f)_{5} \cong H^1(T_X(3)) \cong H^{1,4}(X)$.\end{itemize}
\end{lemma}
\begin{proof}

Follows from Corollary \ref{gr2n}, since neither $\frac{2n-1-d}{3}$ nor $\frac{4n-9-d}{3}$ are integer numbers.
\end{proof}

Now that we established the isomorphisms in abstract, we want to explicitly compute the Griffiths ring of a Gushel-Mukai fivefold.\\
We have therefore to make explicit the action of $\s_5$ on $H^0(\Gr(2,5),\of_G(2))$, the latter being the degree 2 component of the quotient of $\C[x_{1,2}, \ldots , x_{4,5}]$ by the ideal generated by the Pl\"ucker relations.\\
The derivations $D^i_j$ acts as $$D^i_j(x_{r,s} \cdot x_{h,k})=( \delta_{j,r} x_{i,s}+\delta_{j,s} x_{r,i}) x_{h,k}+ x_{r,s}\cdot ( \delta_{j,h} x_{i,k}+\delta_{j,k}x_{h,i}).$$
Extending by linearity we can rewrite the $D^i_j$ in a much more neat form as $$D^i_j= \sum_{k=1}^5 x_{k,i} \frac{\de}{\de x_{k,j}}.$$ 
We prepared a Macaulay2 script that, given a polynomial $f \in H^0(\Gr(2,5),\of_G(2))$ returns the 24 polynomials $D^{i}_j(f)$. The polynomial $f$ needs to be chosen such that the corresponding $X$ is smooth: in turn, this can be checked a posteriori. 
In particular, the Fermat-type polynomial $$f= \sum a_{i,j} x_{i,j}^2$$ works as a choice, as long as we take the coefficients $a_{i,j}$ in a fairly generic way. In particular, none of the $D^j_i(f)$ has to cancel out and become identically zero: to this purpose picking $a_{i,j} \neq a_{r,s}$ will be enough. As an example with random coefficients we can therefore pick $$f=x_{1,2}^2+2x_{1,3}^2+4x_{1,4}^2+5x_{1,5}^2+6x_{2,3}^2+11x_{2,4}^2+75x_{2,5}^2+13x_{3,4}^2+43x_{3,5}^2+8x_{4,5}^2. $$
Using the formula above we write the twenty-four differential polynomials as \begin{align*}
 D^2_1(f)&= 4x_{1,3}x_{2,3}+8x_{1,4}x_{2,4}+10x_{1,5}x_{2,5},     \\    &\vdots                    \\    D^4_4(f)-D^5_5(f)&=8x_{1,4}^2-10x_{1,5}^2+22x_{2,4}^2-150x_{2,5}^2+26x_{3,4}^2-86x_{3,5}^2.
\end{align*}
Denote by $D$ the ideal generated by the 24 polynomials above and $f$. Let $P$ be the ideal generated by the Pl\"ucker equations. By the description above we have $$ R^G_f \cong \C[x_{1,2}, \ldots, x_{4,5}]/(P+D).$$
The Hilbert-Poincaré series of $R^G_f$ is $$\HP(R^G_f)= 1+10t+25t^{2}+10t^{3}+t^4.$$ 
By Lemma \ref{gmhodge} we have $0=(R^G_f)_{-1} \cong H^{4,1}(X)\cong \overline{H^{1,4}(X)}$ and $\C^{10}= (R^G_f)_1 \cong H^{3,2}(X) \cong \overline{H^{2,3}(X)} \cong (R^G_f)_3\cong(R^G_f)_{1}^{\vee}$. This coincides with the calculation already done above.\\
In particular, $(R^G_f)_1 \cong H^{3,2}(X)$ is generated by the degree 1 element in $R$, that is the ten linear forms $\lbrace x_{i,j} \rbrace$, dual to $H^{2,3}(X) \cong (R^G_f)_3$ with respect to the socle generator $x_{4,5}^4$ of $R_4$.

\section{Complete intersections in Grassmannians}
 Let $Z=Z_{d_1, \ldots, d_c} \subset \Gr(k,n)$ be a smooth codimension $c$ complete intersection of multi-degree $d_1, \ldots, d_c$. Set $m=\sum d_i-n$ the adjunction degree of $Z$: in particular, $\omega_Z \cong \of_Z(m)$. Equivalently, $Z$ is defined by a section $\sigma \in H^0(G, \E)$, where $\mathcal{E}= \bigoplus_{i=1}^c \of_G(d_i)$.  
 We associate to $Z$ a hypersurface $\zt \subset Y = \PP(\E)$ with a \emph{Cayley trick} as explained in the preliminaries. This is in fact the same circle of ideas that led to Theorem \ref{dimcares}, allowing to translate the result we obtained in the hypersurface case to complete intersections.
 
 We denote by $\hN=N+(c-1)$ the dimension of $Y$.  The projective bundle $Y$ has Pic$(Y) \cong \Z^2$: pick as a $\Z$-basis $\langle \mL, D \rangle$ with $D=\pi^* \of_G(1)$, and $\mL$ being the tautological quotient line bundle. With respect to this grading, we write  $\mathcal{F}(a,b):= \mathcal{F} \otimes \mL^a \otimes D^b$ and $H^i_{*,*}(\mathcal{F})$ for $\bigoplus_{a,b} H^i(\mathcal{F}(a,b))$. We define the \emph{Griffiths ring} of $Z$ as follows.
\begin{definition} Let $Z, \zt$ be as above. The Griffiths ring of $Z$ is \begin{equation}\label{griffithsring1} \mU=\bigoplus_{a,b}\mU_{a,b}$$ with $$\mU_{a,b}= H^1(\zt, T_{\zt} \otimes \mL^{a-1} \otimes D^b).
\end{equation} \end{definition}
Notice that a priori $\mU$ above has only the structure of bi-graded vector space. The ring structure is given by the following tangent-normal exact sequence, where we denote with $\mL$ the restriction of $\mL$ to $\zt$ as well.
 $$0 \to T_{\zt} \to T_Y|_{\zt} \to \of_{\zt}(\mL) \to 0.$$ 
For any $(a-1,b)$ we consider the twisted version of the above sequence
 $$0 \to T_{\zt}(a-1,b) \to T_Y|_{\zt}(a-1,b) \stackrel{\varphi}{\to} \of_{\zt}(a,b) \to 0.$$
 From the twisted Koszul resolution associated to $\zt$ one can check that $H^1(T_Y|_{\zt}(a,b))= 0$. Therefore, if $F$ denotes the equation of $\zt$ one has  $$H^1(\zt, T_{\zt} \otimes \mL^{a-1} \otimes D^b) \cong H^0(Y, \mL^a \otimes D^b)/(F, \textrm{Im}(\varphi)).$$
 Therefore the ring structure of $\mU$ descends directly from that of  $H^0_{*,*}(Y, \mL^a \otimes D^b)$.  We identify $$\bigoplus_{a,b}H^0(Y, \mL^a \otimes D^b) \cong S[y_1, \ldots, y_c],$$ where $S$ denotes the coordinate ring of the affine cone over the Grassmannian $\Gr(k,n)$. We set the Pl\"ucker variables $x_I$ to have bi-degree (0,1), and the new fiber variables $y_i$ bi-degree $(1, -d_i)$. The choice of bi-grading of the variable is taken in accordance with the projective case, as in Theorem \ref{dimcares}.
 In this set of coordinates,  the equation of $\zt$, $F \in (S[y_1, \ldots, y_c])_{1,0}$, is defined as $F:= \sum_i y_i f_i,$ where the $f_i$ are the equations of the complete intersection $Z$. This is the same strategy used in \cite{konno}, \cite{dimca1994residues} and recalled in Theorem \ref{dimcares}.
  
 From the relative tangent sequence \eqref{reltang} we have that the action of $H^0_{*,*}(T_Y)$ splits into the direct sum of its vertical part and the horizontal part: from the discussion in the previous section, Lemma 2.5, and \cite{konno}, we make explicit this action and give a new definition of the Griffiths Ring of $Z$, that coincides with the one given above.
 \begin{definition} Let $Z$ and $S$ be as above, with the variables $x_I$ with bi-degree (0,1), and the variables $y_i$ with bi-degree $(1, -d_i)$. The \emph{Griffiths ring} of $Z$ can be equivalently defined as \begin{equation}\label{griffithsring2}
  \mU:=S[y_1, \ldots, y_c] /(F, \frac{\de F}{\de y_1},\ldots, \frac{\de F}{\de y_c}, \lbrace D_{x_I}(F) \rbrace ).
 \end{equation}
 \end{definition}
 The derivations $D_{x_I}$ are the ones already defined in the previous section. Notice that $$D_{x_I}(F) = \sum _i y_i D_{x_I}(f_i)$$ and $$\frac{\de F}{\de y_i}=f_i.$$ In turn, the above definition can be further simplified as
  $$\mU:=S[y_1, \ldots, y_c] /(F, f_1, \ldots,f_c, \lbrace D_{x_I}(F) \rbrace ).$$
  In particular, the ideal above can be directly compared with the one in Theorem \ref{dimcares}, where we consider $D_{x_I}(F)$ instead of $\frac{\de F}{\de x_i}$, as in the hypersurface case.   
  
 From the relative Euler sequence we have $\omega_Y \cong \mL^{-c}\otimes D^{m}$,  and by adjunction formula $$\omega_{\zt} \cong \mL^{-c+1}\otimes D^{m}.$$ From  \eqref{griffithsring1} and Proposition \ref{konno} we have the following immediate corollary.
   \begin{corollary}
  $\mU_{1,0} \cong H^1(\zt, T_{\zt}) \cong H^1(Z, T_Z)$.
  \end{corollary}
We are now able to prove the main result of this section. Define $\delta$ and $I$ as in Theorem \ref{residueGrass}. As above $m:= \sum d_i -n$ so that $\omega_Z \cong \of_Z(m)$.
\begin{thm} \label{grci} Let $Z=Z_{d_1, \ldots, d_c}$ be a smooth complete intersection in a Grassmannian Gr(k,n), and let $\mU$ be the Griffiths ring attached to $Z$. Suppose $m \geq -1$. Then if $ \ddim(Z)=N-c$ is even $$ \mU_{p, m} \cong H^{N-c-p,p}_{\van}(Z).$$ If  $ \ddim(Z)=N-c$ is odd $$ \mU_{p, m} \cong H^{N-c-p,p}_{\van}(Z) \oplus \delta_{p, \frac{N-c}{2}}I_{p,p-1}(G).$$
\end{thm}
\begin{proof}
The first step consists in reducing our analysis to the study of $Y$. From Proposition \ref{konno}, it is enough to prove that $$\mU_{p+1-c, m} \cong H^{\hN-1-p,p}_{\van}(\zt).$$
In fact, $$  H^{\hN-1-p,p}_{\van}(\zt) \cong H^{N+c-2-p,p}_{\van}(\zt) \cong H^{N-p-1,p-c+1}_{\van}(Z),$$ and setting $p'=p+1-c$ we obtain the statement.\\
By definition of Griffiths ring, we have therefore to show that $$ \mU_{p+1-c,m}\cong H^1(\zt, T_{\zt} \otimes \omega_{\zt} \otimes \mL^{p-1}) \cong H^1(\zt, \Omega^{\hN-2} \otimes \mL^{p-1}) \cong H^{p}_{\van}(\zt, \Omega^{\hN-1-p}_{\zt}) .$$
The only non-obvious isomorphism is the second one. This is proved inductively as follows. First use the two exact sequences (Koszul and tangent-normal) \begin{equation}\label{seq1}
0 \to \Omega^{k-1}_{\zt} \otimes \mL^{p-1} \to \Omega^k_{Y|_{\zt}} \otimes \mL^p \to\Omega^{k}_{\zt} \otimes \mL^{p} \to 0 
\end{equation}
\begin{equation}\label{seq2}
 0 \to \Omega^k_Y \otimes \mL^{p-1} \to \Omega^k_Y \otimes \mL^p \to \Omega^k_{Y|_{\zt}} \otimes \mL^p \to 0
\end{equation}
From Lemma 4.9 \cite{konno},  the groups $ H^i(Y,\Omega^k_j \otimes \mL^{p-1})$ vanish if $H^r(G, \Omega^s \otimes \textrm{det}( \E) \otimes \Sym^t \E)=0$, for specific values of $r,s,k$. But from Theorem \ref{snow} all these groups vanish when $m \geq -1$. As in the hypersurface case, the only vanishings that are not automatic are for $H^{p,p}(Y)$. Indeed, using \eqref{seq1} and \eqref{seq2} one gets the division in even and odd case, similarly to what we did in Theorem \ref{residueGrass}. Moreover by K\"unneth formula, $I_{p,p-1}(Y)=I_{p,p-1}(G)$. \\  
When these vanishings are not satisfied, the residual contributions depend only on $H^*(\Omega_Y^k \otimes \mL^j)$. These cohomology groups can be expressed in terms of (cohomology of) $\pi^* \Omega_G^k$ and the relative cotangent bundle $\Omega^{c-1}_{Y/G}$ by picking appropriate exterior powers of the short exact sequence
\begin{equation} \label{seq3}
0 \to \pi^* \Omega^1_G \otimes \mL^p \to \Omega^1_Y \otimes \mL^p \to \Omega^1_{Y/G} \to 0.
\end{equation}
Equivalently, as in Lemma 1.4, \cite{konno}, one could use the following spectral sequence $$E^{i,j-1}_1 =H^j(Y, \Omega^{p-i}_{Y/G} \otimes \mL^p \otimes \pi^*(\Omega^i_G \otimes V)) \Rightarrow H^j(Y, \Omega^p_Y \otimes \mL^p \otimes \pi^*V).$$
The last step consists in expressing the cohomology groups of the exterior power of the relative cotangent bundle in terms of the cohomology groups of bundles on the Grassmannian. This is done via the following sequence (sequence (3) in \cite{konno})
\begin{equation}\label{seq4}
0 \to \Omega^l_{Y/G} \otimes \mL^p \otimes\pi^*\Omega_G^{k-1}\to \pi^*(\W^l \E \otimes \Omega^{k-1}_G) \otimes \mL^{p-l} \to \Omega^{l-1} \otimes \mL^p  \otimes \pi^* \Omega^{k-1}_G \to 0.
\end{equation}
Since $Z$ is a complete intersection in Gr(k,n), its normal bundle in the Grassmannian is $\E= \oplus \of_G(d_i)$. Therefore we are in the situation of Lemma \ref{pistar}, and we can express any cohomology group of the form $H^q(Y, \pi^*\Omega^k \otimes \mL^p)$ as a function of either $H^q(G, \Omega_G^k \otimes \Sym^p \E)$ or $H^{q-r+1}(G, \Omega^k_G \otimes det (\E^*) \otimes \Sym^{-p-c} \E^*)$, with both $S^p\E$ and $det(\E^*)$ equal to (the sum of some) $\of_G(d_i)$.  
\end{proof} 
Our Theorem closely mirrors the statement in Theorem \ref{dimcares} in the projective space case. As in that case, the degrees of the relevant bigraded components are only functions of the multi-degree and the canonical class of the complete intersection.

From the proof of the above Theorem we can immediately obtain the following corollary.
\begin{corollary} Let $Z=Z_{d_1, \ldots, d_c}$ be a smooth complete intersection in a Grassmannian Gr(k,n), and let $\mU$ be the Griffiths ring attached to $Z$, and $m=\sum d_i.-n$. Then $$ \mU_{p, m} \oplus B_{N-c-p,p}\cong H^{N-c-p,p}_{\van}(Z) \oplus A_{N-c-p,p},$$ where $A_{N-c-p,p}, B_{N-c-p,p} $ depend only on the residual cohomology groups $H^i(G, \Omega^j_G(k))$ for appropriate values of $i,j,k$. 
\end{corollary}
We will analyse in full detail one example in which actually the residual contributes are not all zero, showing how it is possible to get explicit results without restriction on the degrees.

\subsection{A worked example: a linear section of the Grassmannian $\Gr(2,7)$}
The first example we want to describe in detail is the Calabi-Yau threefold $X_{1^7} \subset \Gr(2,7)$ already famous in literature, see for example \cite{rodland} or \cite{borisov}. In the cited paper, R\o dland computed its Hodge numbers. We compute the full Griffiths ring. Since its canonical class is trivial, Theorem \ref{grci} applies directly. In particular, its Griffiths ring contains the Hodge groups as special homogeneous slices, without any residual contribution from the ambient Grassmannian. Picking the following seven general equations 
\begin{align*}
f_1&=x_{1,2}+2 x_{2,6}+3 x_{3,5},\\
f_2&=x_{1,6}+4 x_{2,5}+5 x_{3,4},\\
f_3&=x_{1,5}+6 x_{2,4}+7 x_{6,7},\\
f_4&=x_{1,4}+8 x_{2,3}+9 x_{5,7},\\ 
f_5&=x_{1,3}+10 x_{4,7}+11 x_{5,6},\\ 
f_6&=x_{1,2}+1 2x_{3,7}+13 x_{4,6},\\
f_7&=x_{3,6}+x_{2,7}+x_{4,5}.
\end{align*}
Denote by $I$ the ideal generated by these seven equations in the coordinate ring $S$ of the Grassmannian Gr(2,7). One can check with a direct computation that the variety defined by this set of equations is smooth. Of course the choice of coefficients is not influential, provided that they are sufficiently general. The action of $\mathfrak{sl}_7$ on the coordinate ring of $X_{1^7}$ is generated by 48 homogeneous degree 1 equations that are easily written down. So, if as before we denote by $$F= \sum_{i=1}^7 y_i f_i,$$ where each $y_i$ has bi-degreee $(1,-1)$ we have that $$ \mU \cong S[y_1, \ldots, y_7]/ (D+F+I).$$
The ideal $D$ is generated by the induced $\mathfrak{sl}_7$ action on the ring $S[y_1, \ldots, y_7]$. We can easily compute the generators which are 
\begin{align*}
D^1_1(F)-D^2_2(F)&=-2x_{2,6}y_1-4x_{2,5}y_2+x_{1,6}y_2-6x_{2,4}y_3+x_{1,5}y_3-8x_{2,3}y_4+x_{1,4}y_4+x_{1,3}y_5,\\   &\vdots     \\  D^6_7(F)&=2 x_{2,7}y_1+x_{1,7}y_2+11 x_{5,7}y_5+13 x_{4,7}y_6.
\end{align*}We compute then the first (graded) components of the Griffiths ring. Their dimensions are
\begin{center}

\begin{tabular}{|c|c|c|c|c|c|c|c|c|}
\hline 
$a/b$ & -4 & -3 & -2 & -1 & 0 & 1 & 2 & 3 \\ 
\hline 
-1 & 0 & 0 & 0 & 0 & 0 & 0 & 0 & $\ldots$ \\ 
\hline 
0 & 0 & 0 & 0 & 0 & 1 & 14 & 70 &210 \\ 
\hline 
1 & 0 & 0 & 0 & 7 & 50 & 91 & 28& 0 \\ 
\hline 
2 & 0 & 0 & 28 & 84 & 51 & 7 & 0 & 0 \\ 
\hline 
3 & 0 & 84 & 77 & 14 & 1 & 0 & 0 & 0 \\ 
\hline 
4 & 210 & 21 & 0 & 0 & 0 & 0 & 0 & 0 \\ 
\hline 
\end{tabular} 
\end{center}

The (vertical) slice with $b=0$ corresponds to the Hodge groups of $X_{1^7}$. In fact, as predicted by Theorem \ref{grci} we have
\begin{align*}
H^{3,0}(X) &\cong \mU_{0,0} \cong \C,\\
H^{2,1}(X) &\cong \mU_{1,0} \cong \C^{50},\\
H^{1,2}(X)  \oplus I_{2,1} &\cong \mU_{2,0} \cong \C^{50}\oplus \C, \\
H^{0,3}(X) &\cong \mU_{3,0} \cong \C.
\end{align*}

\subsection{A worked example: Fano fourfold of degree 10}
We focus now on a smooth complete intersection $Z_{2,1} \subset \Gr(2,5)$. This is a linear section of the fivefold considered in the hypersurface section of this paper. Its Hodge numbers can be found for comparison in \cite{ilievmanivel2}. Again, the purpose of this example is to compute the full Griffiths ring, together with generators and relations. This 4-fold has dimension 4 and canonical class $\of_Z(-2)$. To $Z$ is associated the adjoint 6-fold hypersurface $\zt \subset\PP(\E)$, where $\E= \of_G(1) \oplus \of_G(2)$. Since the index of $Z$ is greater than one, Theorem \ref{grci} does not apply directly. We want to explicitly compute the residual contribute $A_{p, N-p-1}$ and give an explicit presentation for the Griffiths ring $\mU$ associated to $Z$. The main result here is the following:
\begin{proposition} Let $Z, \zt$ be as above. We have the following
\begin{itemize}
\item $H^0(\Omega^4_Z) \cong H^1(\Omega^5_{\zt})\cong \mU_{0,-2}$;
\item $H^1(\Omega^3_Z) \cong H^2(\Omega^4_{\zt}) \cong \mU_{1,-2}$;
\item $H^2_{\van}(\Omega^2_Z) \cong H^3_{\van}(\Omega^3_{\zt}) \cong \mU_{2,-2}/ V_5$;
\item $H^3(\Omega^1_Z) \cong H^4(\Omega^2_{\zt}) \cong \mU_{3,-2}$;
\item $H^4(\of_Z) \cong H^5(\Omega^1_{\zt}) \cong \mU_{4,-2}$.
\end{itemize}
\end{proposition}
\begin{proof}
The first isomorphism of any row follows from Proposition \ref{konno}. We will prove only the first 3 points, the other being analogous and following by duality. Moreover (1) is obvious, since all three terms are equal to zero. So we are left to prove part (2) and (3). We will divide the proof in three separate lemmata.
\begin{lemma} $H^1(T_{\zt} \otimes\omega_{\zt} \otimes \mL) \cong H^2(\Omega^4_{\zt})$.
\end{lemma}
\begin{proof}
We start with the observation that the above lemma proves point (2), since $$H^1(T_{\zt}\otimes \omega_{\zt} \otimes \mL) =: \mU_{1,-2}.$$\\
By tangent pairing, $$H^1(T_{\zt} \otimes\omega_{\zt} \otimes \mL) \cong H^1(\Omega^5_{\zt} \otimes \mL).$$ We use the sequence \eqref{seq1} with $k=5$ and $p=1$. In cohomology this becomes
$$0 \to H^1(\Omega^5_{Y|_{\zt}} \otimes \mL) \to H^1(\Omega^5_{\zt} \otimes \mL) \to H^2(\Omega^4_{\zt}) \to H^2(\Omega^5_{Y|_{\zt}} \otimes \mL) \to 0,$$
with the first and last zeroes given, respectively, by K\"unneth formula and by Akizuki-Kodaira-Nakano vanishing. Using the same arguments, from sequence \eqref{seq2} we immediately get $$ 0 \to H^1(\Omega^5_Y \otimes \mL) \to H^1(\Omega^5_{Y|_{\zt}} \otimes \mL) \to 0 $$ and $$ 0 \to H^2(\Omega^5_Y \otimes \mL) \to H^2(\Omega^5_{Y|_{\zt}} \otimes \mL) \to 0.$$
Consider now sequence \eqref{seq3}. Since the normal bundle to $\zt$ has rank 2, the relative cotangent bundle $\Omega^1_{Y/G}$ is a rank 1 bundle. Therefore the raised relative tangent sequence, when tensored with $\mL$ has a particularly simple form $$0 \to \pi^* \Omega^5_G\otimes \mL \to \Omega^5_Y \otimes \mL \to \pi^*\Omega^4_G  \otimes \Omega^1_{Y/G} \otimes \mL \to 0.$$
By Proposition \ref{konno}, $$H^i(Y, \pi^* \Omega^5_G \otimes \mL) \cong H^i(G,\Omega^5_G(1)) \oplus H^i(G,\Omega^5_G(2)).$$ These groups are all 0 for $i=1,2,3$ (see \cite{peternell}, Lemma 0.1). Therefore $$H^i(\Omega^5_Y \otimes \mL) \cong H^i(\pi^*\Omega^4_G  \otimes \Omega^1_{Y/G} \otimes \mL), \ \ i=1,2.$$
Finally, by sequence \eqref{seq4} $$ 0 \to \Omega^1_{Y/G} \otimes \mL \otimes \pi^*\Omega^4_G \to \pi^*(\Omega^4_G(2) \oplus \Omega^4_G(1)) \to \mL \otimes \pi^* \Omega^4_G \to 0.$$
Using Proposition \ref{konno}, Kodaira vanishing and the Peternell-Wisniewski Lemma we have $$H^j ( \pi^*(\Omega^4_G(2) \oplus \Omega^4_G(1)))=0, \ \ j=0,1,2$$ and $$H^l(\pi^*\Omega^4_G \otimes \mL)=0, \ \ l=1,2.$$ In particular, from all these vanishings $$H^1(\Omega^5_{Y|_{\zt}} \otimes \mL)=H^2(\Omega^5_{Y|_{\zt}} \otimes \mL)=0$$ and the result follows.
\end{proof}
To prove part (3) of the proposition, we need to combine the two following results.
\begin{lemma}\label{lem1app} $H^2(\Omega^4_{\zt} \otimes \mL) \cong H^3_{\van}(\Omega^3_{\zt})$
\end{lemma}
\begin{lemma} \label{lem2app}$H^1(\Omega^5_{\zt} \otimes \mL^2) \cong H^2(\Omega^4_{\zt} \otimes \mL) \oplus V_5$
\end{lemma}
The two Lemma above together prove the result, since $$\mU_{2,-2}= H^1(T_{\zt} \otimes \omega_{\zt} \otimes \mL^2)\cong H^2(\Omega^4_{\zt} \otimes \mL) \oplus V_5\cong H^3_{\van}(\Omega^3_{\zt}),$$ as required.
\begin{proof}[Proof of Lemma \ref{lem1app}]
We use the same tools of the previous Lemma. The first step is the reduction to $$ 0 \to H^2(\Omega^4_{Y|_{\zt}} \otimes \mL) \to H^2(\Omega^4_{\zt} \otimes \mL) \to H^3(\Omega^3_{\zt}) \to H^3(\Omega^4_{Y|_{\zt}} \otimes \mL) \to 0 $$
Then, since by K\"unneth formula $H^3(\Omega^4_Y)=0$ we consider the two induced sequences
$$ 0 \to H^2(\Omega^4_Y \otimes \mL) \to H^2(\Omega^4_{Y|_{\zt}} \otimes \mL) \to 0,$$
$$0 \to H^3(\Omega^4_Y \otimes \mL) \to H^3(\Omega^4_{Y|_{\zt}} \otimes \mL)\to H^4(\Omega^4_Y) \to 0.$$
From sequences \eqref{seq3}, \eqref{seq4} we get the vanishings of $H^2(\Omega^4_Y \otimes \mL)$ and $H^3(\Omega^4_Y \otimes \mL)$. This implies
$$ 0 \to H^2(\Omega^4_{\zt} \otimes \mL) \to H^3(\Omega^3_{\zt}) \to H^4(\Omega^4_Y) \to 0,$$
and therefore by definition and Lesfchetz hyperplane section theorem $$H^2(\Omega^4_{\zt} \otimes \mL) \cong H^3_{\van}(\Omega^3_{\zt}).$$
 
 The contribution of $H^4(\Omega^4_Y)$ can be easily computed from the K\"unneth formula: in fact  $$H^4(\Omega^4_Y) \cong H^4(Y, \C) \cong H^4(\Gr(2,5)) \otimes H^0(\PP^1) \oplus H^3(\Gr(2,5)) \otimes H^1(\PP^1) \oplus H^2(\Gr(2,5)) \otimes H^2 (\PP^1).$$
In particular, $H^3(\Gr(2,5)) \cong \C^3$ and $H^2(\Gr(2,5)) \cong \C^2$ and therefore $H^4(\Omega^4_Y) \cong \C^5$.

\end{proof}
\begin{proof}[Proof of Lemma \ref{lem2app}]
The first thing that we need to show is $H^1(\Omega^4_{\zt} \otimes \mL)=0$. By using sequences \eqref{seq1} and \eqref{seq2} this is equivalent to showing that $H^1(\Omega^4_Y \otimes \mL)=0$. This is implied by sequences \eqref{seq3}, \eqref{seq4} together with $H^0(\Omega^3_G(1))= H^0(\Omega^3_G(2))=0$, see \cite{peternell}. Therefore we have 
\begin{equation}\label{lemma34app}
0 \to H^1(\Omega^5_{Y|_{\zt}} \otimes \mL) \to H^1(\Omega^5_{\zt} \otimes \mL^2) \to H^2 (\Omega^4_{\zt}\otimes \mL) \to H^1(\Omega^5_{Y|_{\zt}} \otimes \mL) \to 0.
\end{equation}
On the other hand from the residue sequence \eqref{seq2} $$H^1(\Omega^5_{Y|_{\zt}} \otimes \mL) \cong H^1(\Omega^5_Y \otimes \mL^2), \ \ H^2(\Omega^5_{Y|_{\zt}} \otimes \mL) \cong H^2(\Omega^5_Y \otimes \mL^2).$$
Set $\mathcal{M}:=\pi^*\Omega^4_G \otimes \Omega^1_{Y/G} \otimes \mL^2$. We see from sequence \eqref{seq3} $$ H^1(\Omega^5_Y \otimes \mL^2)\cong H^1(\mathcal{M}), \ \  H^2(\Omega^5_Y \otimes \mL^2) \cong H^2(\mathcal{M}).$$
From \eqref{lemma34app} we have $$H^1(\Omega^5_{\zt} \otimes \mL^2) \cong H^2(\Omega^4_{\zt} \otimes \mL) \oplus H^1(\mathcal{M})/H^2(\mathcal{M}).$$
\cite[Lemma 1.5, ii]{konno} gives $H^0(\mathcal{M})=0$. By Borel-Bott-Weil $$H^0(\pi^*(\Omega^4_G(1)\oplus \Omega^4_G(2)) \otimes \mL)\cong H^0(\mL^2 \otimes \pi^* \Omega^4_G) \cong \mathcal{V}^{2},$$ with the latter denoting the unique irreducible SL(5)-module of highest weight -2. Moreover 
$$H^1(\pi^*(\Omega^4_G(1)\oplus \Omega^4_G(2)) \otimes \mL) \cong V_5 \oplus V_5 \cong \C^{10}$$ and  $$H^1(\pi^*\Omega^4_G \otimes \mL^2) \cong V_5.$$
Therefore by sequence \eqref{seq4} $H^1(\mathcal{M})/H^2(\mathcal{M}) \cong V_5$, proving the Lemma.
\end{proof}

\end{proof}
We now construct explicitly the Griffiths ring $\mU$. The ambient ring $S[y_1, y_2]$ is the Pl\"ucker ring already constructed in the previous section with the two new variables $y_1, y_2$ added. The variables $x_{i,j}$ have bi-degree $(0,1)$ while $y_1$ and $y_2$ have bi-degree (respectively) $(1,-1)$ and $(1,-2)$. As a quadric we choose the same one of the hypersurface case, that is $$f_2=x_{1,2}^2+2x_{1,3}^2+4x_{1,4}^2+5x_{1,5}^2+6x_{2,3}^2+11x_{2,4}^2+75x_{2,5}^2+13x_{3,4}^2+8x_{4,5}^2+43x_{3,5}^2$$
while as a linear equation we pick $$f_1=x_{1,2}+x_{3,4}.$$ We remark that the latter equation defines a smooth hypersurface of $\Gr(2,n)$ only when $ n\leq 5$. The 24 derivations are obtained easily  from the formula $\sum y_i D_x(f_i)$, given that we already know how each of the infinitesimal derivations in $D_x(f_i)$ acts from the hypersurface example.
For example  
\begin{align*}
 D^2_1(F)&=y_1(4x_{1,3}x_{2,3}+8x_{1,4}x_{2,4}+10x_{1,5}x_{2,5} ), \\    &\vdots                    \\    D^4_4(F)-D^5_5(F)&=y_1(8 x_{1,4}^2-10x_{1,5}^2+22x_{2,4}^2-150x_{2,5}^2+26x_{3,4}^2-86x_{3,5}^2)+y_0(x_{3,4})
\end{align*}
Denote by $D$ the ideal generated by all these derivations. We have $$\mU= S[y_1,y_2]/(D, F, f_1, f_2).$$ We compute some of the graded components of $\mU$
\begin{center}

\begin{tabular}{|c|c|c|c|c|c|c|c|c|}
\hline 
a/b & -4 & -3 & -2 & -1 & 0 & 1 & 2 & 3 \\ 
\hline 
-1 & 0 & 0 & 0 & 0 & 0 & 0 & 0 & $\ldots$ \\ 
\hline 
0 & 0 & 0 & 0 & 0 & 1 & 10 & 50 & $\ldots$ \\ 
\hline 
1 & 0 & 0 & 1 & 10 & 24 & 10 & 1 & 0 \\ 
\hline 
2 & 1 & 10 & 25 & 10 & 1 & 0 & 0 & 0 \\ 
\hline 
3 & 25 & 11 & 1 & 0 & 0 & 0 & 0 & 0 \\ 
\hline 
4 & 2 & 0 & 0 & 0 & 0 & 0 & 0 & 0 \\ 
\hline 
\end{tabular} 
\end{center}

In particular, $\mU_{1,0} \cong H^1(T_Z)$,  $\mU_{-1,-2}= H^{4,0}(Z)$, $\mU_{1, -2} \cong H^{3,1}(Z),$ $\mU_{2,2}= V_5 \oplus H^{2,2}_{\van}(Z)$, $\mU_{3,-2}= H^{1,3}(Z)$ and $\mU_{4,-2}=H^{0,4}(Z)$. The Hodge diamond of $Z=Z_{2,1}$ is then
\[ \begin{matrix}

&0 &&1 && 22 &&1 && 0 &\\
&& 0 && 0 && 0 &&0 \\
&&&0 &&1&&0&&&\\
&&&&0&&0&&&&\\
&&&&&1 &&&&&
\end{matrix}\]
and one can compare with the results in \cite{kuznetsovGM}.

\section{Appendix: Fano varieties of K3 type}
Recall the following definition (slightly adapted) from \cite{ilievmanivel}:
\begin{definition} Let $X$ be a smooth projective variety of dimension $n$, such that $h^{i,0}(X)=0$, $i <n$. Define $h:=\lfloor \frac{n-k}{2}\rfloor $. We say that $X$ is of \emph{weak $k$-Calabi-Yau type} if its middle dimensional Hodge structure is numerically similar to a Calabi-Yau $k$-fold, that is
$$h^{n-h,h  }=1,\ \   h^{n-h+j,h-j}=0, \ j\geq 1.$$
We say that $X$ is of \emph{strong $k$-Calabi-Yau type} (or simply of \emph{$k$-Calabi-Yau type}) if in addition the contraction with any generator $\omega \in H^ {n-h,h  }(X)$ induces an isomorphism $$ \omega: H^1(T_X) \longrightarrow H^{n-h-1,h+1}(X).$$
\end{definition}
The case of 3-Calabi-Yau is investigated in \cite{ilievmanivel}. We are particularly interested in the 2-Calabi-Yau case, that is, K3 type. Known examples of these varieties in the strong sense include a smooth cubic fourfold $X_3 \subset \PP^5$, a linear section $Y_1 \subset \Gr(3,10)$, cf. \cite{debarrevoisin}, and in the weak sense the already mentioned Gushel-Mukai fourfold and the c5- K\"uchle variety, cf. \cite{kuznetsovc5}. Some more examples are found if we allow mild singularities - e.g. cyclic quotient - see \cite{frz19}. Most of these examples are deeply linked  with hyperk\"ahler geometry and derived category problems. Moreover by \cite{kuznetsovmarkushevic} families of Fano of K3 type (FK3) are likely to be linked with projective families of irreducible holomorphic symplectic manifolds.\\
These families of FK3 necessarily have to be of dimension greater or equal than four and comparatively high index. This implies we have to apply Theorem \ref{grci} with caution, since there may be some residual contributions from the ambient space to take into account. However, there is some good news. Denote by $T= \C[ {x_I}, {y_i}]$ the basic ambient ring from which we build the Griffiths ring $\mU$, suitably bigraded as in \eqref{griffithsring1}. For a variety $X$ of dimension $2s$, a sub-structure of K3-type implies $h^{s+1,s-1}(X)=1$ and $h^{s+t, s-t}(X)=0$, for $t>1$. We therefore look at $\mU_{i, m}$, with $i \leq s-1$ having the above numerological properties. Since the relations in the Griffiths ring $\mU$ are all in bidegree $(0,1)$ and $(1,0)$, $m$ is negative and we have for $i$ in such a range that $T_{i,m}=\mU_{i,m}$. This reduces the problem into a combinatorial one.\\
Let in fact $X$ be a complete intersection of index $m$ in the Grassmannian $\Gr(k,l+k)$ given by the bundle $\mathcal{F}= \bigoplus \of_G(d_i)$. Denote by $\alpha=c_1(\mathcal{F})= \sum d_i$. A quick analysis of the polynomial ring $T$ reveals that in order to have $$T_{s-1,m}= \C, \ T_{s-t,m}=0$$ the weights must be ordered as $$d_1 > d_2 \geq \ldots \geq d_c$$ and moreover the following equation needs to be satisfied \begin{equation}\label{emagic}2(k+l-\alpha)=d_1(kl-c-2).
\end{equation}
A computer search confirms that only the already mentioned $X_{2,1} \subset \Gr(2,5)$ and $Y_1 \subset \Gr(3,10)$ satisfy this relation. They are the well known Gushel-Mukai fourfold and the Debarre-Voisin Fano 20-fold.\\
However, this does not rule out any other option. Thanks to the residual contributions from the Grassmannian there might be some $X_{d_1, \ldots, d_c}$ with $\mU_{s-1,m} \neq \C$ but still $h^{s-1,s+1}=1$. The condition on the ordering of the weights here might be not required.  This is particularly true in the case of linear sections. Indeed, after a first analysis on the cohomology groups of the ambient Grassmannian, we found another example as $ X_{1^4} \subset \Gr(2,8).$
This is a Fano 8-fold with middle Hodge structure of K3 type. We believe it could lead to a construction of a family of hyperk\"ahler varieties of K3$^{[n]}$ type. We compute its Hodge numbers as  
\begin{proposition}\label{hn} Let $X_{1,1,1,1} \subset \Gr(2,8)$ be given by a generic section of $\of_G(1)^{\oplus 4} $. The Hodge diamond of $X_{1,1,1,1}$ is

\begin{center}
\[\begin{matrix}
0 &&0 &&0&&1 && 22 &&1 &&0 &&0 &&0\\
&0 && 0 && 0 && 0&& 0 &&0 &&0 && 0&\\
&&0 && 0 && 0 &&2 &&0 &&0 && 0&&\\
&&&0&&0 & & 0 & &0 && 0 && 0 &&&\\
&&&&0 &&0 && 2 &&0 && 0 &&&&\\
&&&&& 0 && 0 && 0 &&0&&& \\
&&&&&&0 &&1&&0&&&&&&\\
&&&&&&&0&&0&&&&&&&\\
&&&&&&&&1 &&&&&&&&
\end{matrix}\]
\end{center}
with $h^{4,4}_{\van}(X)=19$.
\end{proposition}
 Notice that the projective dual of $\Gr(2,8)$ is a singular quartic hypersurface in $\PP^{27}$. Cutting the Grassmannian and the quartic with orthogonal linear subspaces we can link $X_{1,1,1,1} \subset \Gr(2,8)$ to a quartic K3 surface $S \subset \PP^3$. An embedding of the derived category of the quartic K3 inside the derived category of the above linear section is provided in \cite{segalthomas}, Thm 2.8.
However, we believe that this could be the only exception. Namely, we make the following
\begin{conj} Let $X=X_{d_1, \ldots, d_c} \subset \Gr(k,n)$ be a Fano smooth complete intersection of even dimension (that is not a cubic fourfold). Then $X$ is not of K3-type unless $$(\lbrace d_i \rbrace, k,n)=(\lbrace 2,1\rbrace,2,5), (\lbrace 1,1,1,1\rbrace,2,8), (\lbrace 1 \rbrace, 3,10).$$ 
\end{conj}
Our method above can be partially extended to more general vector bundles on other homogeneous varieties. In \cite{eg2} we  analysed a handful more examples and study in details their geometric properties.

\end{document}